\documentclass{article}

\usepackage[latin2]{inputenc}
\usepackage{t1enc, anysize, enumerate, amsthm, amsmath, amssymb, epsfig}
\usepackage[unicode]{hyperref}
\newtheorem{theorem}{Theorem}[section]
\newtheorem{lemma}[theorem]{Lemma}

\newtheorem{remark}[theorem]{Remark}

\newenvironment{definition}[1][Definition]{\begin{trivlist}
\item[\hskip \labelsep {\bfseries #1}]}{\end{trivlist}}

\newenvironment{notation}[1][Notation]{\begin{trivlist}
\item[\hskip \labelsep {\bfseries #1}]}{\end{trivlist}}

\newcommand{\mcf}{\mathcal{F}}
\newcommand{\mcc}{\mathcal{C}}
\newcommand{\mcd}{\mathcal{D}}

\newcommand{\mcr}{\mathcal{R}}
\newcommand{\mcs}{\mathcal{S}}
\newcommand{\mcg}{\mathcal{G}}
\newcommand{\mcp}{\mathcal{P}}

\begin{document}

\title{Forbidden subposet problems with size restrictions}
\author{D\'{a}niel T. Nagy\footnote{E\"{o}tv\"{o}s Lor\'and University, Budapest. dani.t.nagy@gmail.com}}
\maketitle

\begin{abstract}
Upper bounds to the size of a family of subsets of an $n$-element set that avoids certain configurations are proved. These forbidden configurations can be described by inclusion patterns and some sets having the same size. Our results are closely related to the forbidden subposet problems, where the avoided configurations are described solely by inclusions.
\end{abstract}

\section{Introduction}

In this paper, a generalization of the forbidden subposet problem is discussed. Before getting to this generalization, let us overview the original problem. We will use the notation $[n]=\{1,2,\dots, n\}$.

\begin{definition}
Let $\mathcal{P}$ be a finite poset (partially ordered set) with the relation $<_p$. Let $f$ be a function that maps the elements of $\mathcal{P}$ to subsets of $[n]$. We say that $f$ is an {\it embedding} of $\mathcal{P}$ if it is an injective function that satisfies $f(a)\subset f(b)$ for all $a<_p b$. Similarly, $f$ is called an {\it induced embedding} if it is an injective function such that $f(a)\subset f(b)$ if and only if $a<_p b$.
\end{definition}

\begin{definition}
Let $\mathcal{P}_1, \mathcal{P}_2, \dots \mathcal{P}_k$ be finite posets. La($n, \{\mathcal{P}_1, \dots \mathcal{P}_k\}$) denotes the size of the largest family $\mathcal{F}$ of subsets of $[n]$ such that none of the posets $\mathcal{P}_i$ can be embedded into $\mathcal{F}$. Similarly, La$^\star(n, \{\mathcal{P}_1, \dots \mathcal{P}_k\}$) denotes the size of the largest family $\mathcal{F}$ of subsets of $[n]$ such that none of the posets $\mathcal{P}_i$ has an induced embedding into $\mathcal{F}$. (In most problems, we have only one forbidden poset, and we write  La($n, \mathcal{P}$) or La$^\star(n, \mathcal{P}$).)
\end{definition}

The goal in the forbidden subposet problem is to exactly or asymptotically determine the value of the functions La($n, \mathcal{P}$) and La$^\star(n, \mathcal{P}$) for as many posets as possible. There is no general theorem that applies to all posets. However, it is conjectured by all involved researchers that for all $\mcp$, the value of the limit
$$\lim_{n\rightarrow\infty} \frac{La(n,\mcp)}{\binom{n}{\lfloor\frac{n}{2}\rfloor}}$$
is an integer. In all solved cases, the extremal families consist of sets whose sizes are as close to $\frac{n}{2}$ as possible. The problem is asymptotically solved for posets whose Hasse diagram is a tree. (See \cite{bukh} for the noninduced problem and \cite{indtree} for the induced problem.) Upper bounds were given to La($n, \mathcal{P}$), depending on $|\mcp|$ and the length of the longest chain in $\mcp$ \cite{burcsi} \cite{chen} \cite{grosz}.

Roughly speaking, these forbidden poset problems ask for the maximal size of a set family without a configuration (or configurations) that can be described entirely by inclusion. In this paper we consider problems where there are two types of conditions in the forbidden configuration(s): inclusion and certain subsets being required to have the same size. In the next section, we prove many such results and compare them to their counterparts without size restrictions. In the last section, a general theorem is proved. It states that for any such forbidden configuration $S$ there exists a number $C$ such that $|\mcf|\le C\binom{n}{\lfloor\frac{n}{2}\rfloor}$ holds for any family $\mcf$ of subsets of $[n]$ that avoids $S$.

Counting via chains is an essential method to deal with these kind of problems. In the rest of this section, we overview the basics of this technique.

\begin{notation}
Let $A\subset B$ two sets. A \it{chain} between $A$ and $B$ is a family of sets $A=C_{|A|}\subset C_{|A|+1}\subset \dots \subset C_{|B|-1}\subset C_{|B|}=B$, where $|C_i|=i$ for all $|A|\le i\le |B|$. When we say "all chains of $[n]$" we mean the chains between $\emptyset$ and $[n]$. (There are $n!$ such chains.)
\end{notation}

\begin{notation}
$$\Sigma (n,k)=\displaystyle\sum_{i=\lceil\frac{n-k}{2}\rceil}^{\lceil\frac{n+k}{2}\rceil-1} {n \choose i}$$
denotes the sum of the $k$ largest binomial coefficients belonging to $n$.
\end{notation}

\begin{notation}
Let $\mcf$ be a family of subsets of $[n]$. The {\it Lubell function} of $\mcf$ is defined as
$$\lambda(\mcf)=\sum_{F\in\mcf} \binom{n}{|F|}^{-1}.$$
(The name refers to Lubell's proof of Sperner's theorem \cite{Lub}.)
\end{notation}

Since a set $F$ appears in $|F|!(n-|F|)!$ chains out of $n!$, the probability of it being in a random chain is $\binom{n}{|F|}^{-1}$. Denoting the set of all chains of $[n]$ by $\mcc$, the expected number of the elements of $\mcf$ in a random chain is
\begin{equation}\label{lubeq}
\underset{c\in \mcc}{ave}(|c\cap\mcf|)=\sum_{F\in\mcf} \binom{n}{|F|}^{-1}=\lambda(\mcf).
\end{equation}

\begin{lemma} \label{lub2}
Let $\mcf$ be a family of subsets of $[n]$. Then $|\mcf|\le\lambda(\mcf)\binom{n}{\lfloor\frac{n}{2}\rfloor}$.
\end{lemma}

\begin{proof}
$$\lambda(\mcf)=\sum_{F\in\mcf} \binom{n}{|F|}^{-1}\ge \sum_{F\in\mcf} \binom{n}{\lfloor\frac{n}{2}\rfloor}^{-1}=|\mcf|\binom{n}{\lfloor\frac{n}{2}\rfloor}^{-1}.$$
\end{proof}

\begin{lemma} \label{lub}
Let $\mcf$ be a family of subsets of $[n]$. Assume that $\lambda(\mcf)=x+y$, where $x\in\mathbb{N}$ and $y$ is a non-negative real number. Then
$$|\mcf|\le\Sigma(n,x)+y\binom{n}{\lceil\frac{n+x}{2}\rceil}.$$
\end{lemma}

\begin{proof}
For a fixed $|\mcg|$, the value of $\lambda(\mcg)$ is minimal when the sizes of the sets are as close to $\frac{n}{2}$ as possible. Assume that $|\mcf|>\Sigma(n,k)+r\binom{n}{\lceil\frac{n+k}{2}\rceil}$. Select $\Sigma(n,k)$ sets from $\mcf$ such that their sizes are as close to $\frac{n}{2}$ as possible. Then the Lubell function corresponding to their family is at most $k$. The sizes of all remaining sets are at least $\lceil\frac{n+k}{2}\rceil$ or at most $\lfloor\frac{n-k}{2}\rfloor$. Therefore the Lubell function corresponding to their family is at least $(|\mcf|-\Sigma(n,k))\binom{n}{\lceil\frac{n+k}{2}\rceil}^{-1}>
r\binom{n}{\lceil\frac{n+k}{2}\rceil}\binom{n}{\lceil\frac{n+k}{2}\rceil}^{-1}=r$. So $\lambda(\mcf)>k+r$, a contradiction.
\end{proof}

\section{Results}
In this section we prove upper bounds on the sizes of families avoiding certain configurations of inclusion and size restrictions. The original versions of these problems (having only inclusion restrictions) are shown before each problem.

The following simple inequalities will be be used in several proofs in this section, so they are proved separately here.

\begin{notation}
For $k\ge 2$, let $q(k)=\displaystyle\sum_{i=1}^{k-1} \binom{k}{i}^{-1}$.
\end{notation}

\begin{lemma}\label{qlemma}
\begin{enumerate}[i)]
\item $q(k)<\frac{4}{k}$ holds for all $k\ge 2$.
\item $q(k)\le \frac{2}{3}$ holds for all $k\ge 2$, with equality at $k=3$ and $k=4$.
\item If $a,b\ge 2$ and $a+b\ge 13$, then $q(a)+q(b)\le 1$.
\end{enumerate}
\end{lemma}

\begin{proof}
\begin{enumerate}[i)]
\item $$q(k)=\sum_{i=1}^{k-1} \binom{k}{i}^{-1}\le \frac{1}{k}+\frac{1}{k}+(k-3)\frac{2}{k(k-1)}<\frac{4}{k}.$$

\item The statement can be checked manually for $2\le k\le 5$, and follows from part i) for $k\ge 6$.
\item The statement can be checked easily with a computer for $13\le a+b\le 23$. Assume that $a+b\ge 24$, and $a\le b$. Then  part i) implies $q(b)\le\frac{4}{b}\le\frac{4}{12}=\frac{1}{3}$ and part ii) implies $q(a)\le \frac{2}{3}$.
\end{enumerate}
\end{proof}

The following classic theorem provides an upper bound to the size of families avoiding two 3-element posets. Let $V$ denote the 3-element poset with the relations $A<B,C$ and let $\Lambda$ denote the 3-element poset with the relations $B,C<A$.

\begin{theorem}({\bf Katona-Tarj\'an} \cite{tarjan})\label{tarjanthm}
$$La(n,V,\Lambda)=2\binom{n-1}{\lfloor\frac{n-1}{2}\rfloor}.$$
\end{theorem}

The following construction shows that there is a family of size $2\binom{n-1}{\lfloor\frac{n-1}{2}\rfloor}$ that avoids both $V$ and $\Lambda$:
$$\mcf=\left\{F\in[n]~\big|~|F|=\left\lfloor\frac{n}{2}\right\rfloor,~1\not\in F\right\}\cup
\left\{F\in[n]~\big|~|F|=\left\lceil\frac{n}{2}\right\rceil, 1\in F\right\}.$$

We prove that the same bound applies if the forbidden configuration includes two of the sets to having the same size. (The construction obviously works in this case too, so the bound is best possible.)

\begin{theorem}\label{tarjanthm2}
Let $\mcf$ be a family of subsets of $[n]$, where $n\ge 3$. Assume that there are no 3 different subsets in $\mcf$ such that
\begin{enumerate}[a)]
\item $A\subset B$, $A\subset C$ and $|B|=|C|$ or
\item $B\subset A$, $C\subset A$ and $|B|=|C|$.
\end{enumerate}
Then $|\mcf|\le 2\binom{n-1}{\lfloor\frac{n-1}{2}\rfloor}$.
\end{theorem}

\begin{proof}
The statement of the theorem can be checked easily for $n=3$. From now on, we will assume that $n\ge 4$.

It is enough to prove the theorem for even values of $n$, as it follows from $n=2m$ to $n=2m+1$. To see this, assume that we already proved it for even values, and consider an odd $n$. Let $\mcf$ be a sets of subsets of $[n]$, satisfying the conditions of the theorem. Then let
$$\mcf^- = \{F ~|~ F\in\mcf,~1\not\in F\},$$
$$\mcf^+ = \{F\backslash\{1\} ~|~F\in\mcf,~ 1\in F\}.$$

Then $\mcf^-$ and $\mcf^+$ are both families of subsets of $[n-1]$, satisfying the conditions of the theorem. Since $n-1$ is even, their sizes are at most $2\binom{n-2}{\lfloor\frac{n-2}{2}\rfloor}=\binom{n-1}{\frac{n-1}{2}}$. Therefore
$$|\mcf|=|\mcf^-|+|\mcf^+|\le 2\binom{n-1}{\frac{n-1}{2}}.$$

From now on, we will assume that $n$ is even, and use the notation $m=\frac{n}{2}$. Note that for even $n$, $2\binom{n-1}{\lfloor\frac{n-1}{2}\rfloor}=\binom{n}{m}$.

We can assume that $\emptyset,[n]\not\in\mcf$, since $\emptyset\in\mcf$ or $[n]\in\mcf$ would imply that all subsets in $\mcf$ have different size, therefore $|\mcf|\le n+1\le \binom{n}{m}$.

The main idea of the proof is the following. For all sets $F\in\mcf$, we will create a collection of chains called $\alpha(F)$. These collections will be pairwise disjoint, so $F\not=F'\Rightarrow \alpha(F)\cap\alpha(F')=\emptyset$. These collections will be defined such that $|\alpha(F)|\ge (m!)^2$ for all $F\in\mcf$. This will imply $|\mcf|\le \frac{n!}{(m!)^2}=\binom{n}{m}$.

For all $F\in\mcf$, let $\alpha(F)$ consist of all chains that contain $F$, and among the elements of $\mcf$ in the chain, $F$'s size is the closest to $m+\frac{1}{3}$. This way, all chains that contain at least one element of $\mcf$ are added to exactly one of the collections $\alpha(F)$.

Now we give a lower bound to $|\alpha(F)|$.

If $|F|=m$, then all chains passing through $F$ will be added to $\alpha(F)$, therefore $|\alpha(F)|=(m!)^2$.

Assume that $m<|F|<n$. There are $|F|!(n-|F|)!$ chains passing through $F$. All of them are in $\alpha(F)$ except for those that contain a set $G$ satisfying $n-|F|+1\le|G|\le |F|-1$ and $G\subset F$. The conditions of the theorem imply that there are at most $2|F|-n-1$ such sets (one for every possible size). The number of chains passing through both $G$ and $F$ is
$$|G|!(|F|-|G|)!(n-|F|)!\le (|F|-1)!(n-|F|)!.$$
Therefore
$$|\alpha(F)|\ge|F|!(n-|F|)!-(2|F|-n-1)(|F|-1)!(n-|F|)!=$$
$$(n+1-|F|)(|F-1)!(n-|F|)!=(|F|-1)!(n+1-|F|)!\ge (m!)^2.$$

Now assume that $1\le|F|<m$. There are $|F|!(n-|F|)!$ chains passing through $F$. All of them are in $\alpha(F)$ except for those that contain a set $G$ satisfying $|F|+1\le|G|\le n-|F|$ and $F\subset G$. The conditions of the theorem imply that there are at most $n-2|F|$ such sets (one for every possible size). The number of chains passing through both $F$ and $G$ is
$$|F|!(|G|-|F|)!(n-|G|)!.$$
Therefore
\begin{equation}\label{alphaineq}\alpha(F)\ge|F|!(n-|F|)!-\sum_{i=|F|+1}^{n-|F|} |F|!(i-|F|)!(n-i)!.\end{equation}

Assume that $1\le|F|\le m-2$. Consider the sum $\displaystyle\sum_{i=|F|+1}^{n-|F|} (i-|F|)!(n-i)!$. It has $n-2|F|$ summands, all of which are at most $(n-|F|-1)!$. It is also easy to check that $2!(n-|F|-2)!+3!(n-|F|-3)!\le (n-|F|-1)!$. Therefore
$$\sum_{i=|F|+1}^{n-|F|} (i-|F|)!(n-i)!\le (n-2|F|-1)(n-|F|-1)!.$$

It implies by (\ref{alphaineq}) that for $1\le|F|\le m-2$ we have
$$\alpha(F)\ge |F|!\big((n-|F|)!-(n-2|F|-1)(n-|F|-1)!\big)=(|F|+1)!(n-|F|-1)!\ge (m!)^2.$$

So far, we proved that $\alpha(F)\ge (m!)^2$ if $|F|\not= m-1$. Now, assume that $|F|=m-1$. There are $(m-1)!(m+1)!$ chains passing through $F$. All of them are in $\alpha(F)$, with the exception of those that contain a set of $\mcf$ whose size is $m$ or $m+1$. Note that there can be at most 1 set of size $m$ and 1 set of size $m+1$ in $\mcf$ that contains $F$.

If $|G|=m$, $G\in\mcf$ and $F\subset G$, then there are $(m-1)!m!$ chains containing both $F$ and $G$. These chains will not be in $\alpha(F)$. Similarly, if $|H|=m+1$, $H\in\mcf$ and $F\subset H$, then there are $2((m-1)!)^2$ chains containing both $F$ and $G$. These chains will also not be in $\alpha(F)$.

It is easy to see that $|\alpha(F)|\ge (m!)^2$ holds, unless there are sets from $\mcf$ of size both $m$ and $m+1$ containing $F$. If there would be only one of them, then we would have
$$|\alpha(F)|\ge (m-1)!(m+1)!-\max\left((m-1)!m!,~ 2((m-1)!)^2\right)=(m-1)!(m+1)!-(m-1)!m!=(m!)^2.$$

To complete the proof, we need to add some additional chains to the collections corresponding to these elements, so they get at least $(m!)^2$ chains too. Since we already used all chains passing through an element from $\mcf$, we have to use those chains that have no common element with $\mcf$.

Let $F_1, F_2,\dots ,F_p$ denote the sets of size $m-1$ in $\mcf$ that got less than $(m!)^2$ chains assigned to them. For all $F_i$, there are two sets $G_i$ and $H_i$ such that $F_i\subset G_i$, $F_i\subset H_i$, $|G_i|=m$ and $|H_i|=m+1$. The conditions of the theorem imply that if $i\not= j$, then $H_i\not=H_j$ and $F_i\not\subset H_j$. For a fixed $i$, there are two subsets of size $m$ that contain $F_i$ and are contained in $H_i$. If these two sets are different from $G_i$, color both of them red. If one of them is $G_i$, color the other one red. We will call this/these set(s) the red set(s) corresponding to $F_i$. Note that the conditions of the theorem imply that the red sets corresponding to different indices $i$ are different sets.

There are no two subsets of the same size in $\mcf$ that contain a red set, as they would form a forbidden configuration with the corresponding $F_i$. Similarly there are no two subsets of the same size in $\mcf$ that are contained in a red set, as they would form a forbidden configuration with corresponding $H_i$.

Let $X$ be a fixed red set such that $F_i\subset X\subset H_i$. The total number of chains passing through $X$ is $(m!)^2$. If $T\subset X$, then the number of chains between $\emptyset$ and $X$, passing through $T$ is $|T|!(m-|T|)$. Similarly, If $X\subset S$, then the number of chains between $X$ and $[n]$, passing through $S$ is $(|S|-m)!(n-|S|)!$. Therefore the total number of chains passing through $X$ and avoiding $\mcf$ is at least
$$\left(m!-\sum_{i=1}^{m-1} i!(m-i)!\right)^2=
(m!)^2\left(1-\sum_{i=1}^{m-1}\frac{i!(m-i)!}{m!}\right)^2=
(m!\cdot(1-q(m)))^2\ge ((m-1)!)^2.$$
(In the last step, we used that $m(1-q(m))\ge 1$ holds for for $m\ge 2$. It follows easily from Lemma \ref{qlemma} ii).) Let us add the chains passing through a red set corresponding to $F_i$ and avoiding $\mcf$ to $\alpha(F_i)$.

If $G_i\not\subset H_i$, then the original size of $\alpha(F_i)$ can be calculated by taking the number of all chains that are passing through $F_i$ and subtracting those that are passing through $G_i$ or $H_i$ as well. It gives us
$(m-1)!(m+1)!-(m-1)!m!-2((m-1)!)^2=(m!)^2-2((m-1)!)^2$.
In this case, there are two red sets corresponding to $F_i$, so at least $2((m-1)!)^2$ chains are added to $\alpha(F_i)$, making the total number at least $(m!)^2$.

If $G_i\subset H_i$, then the original size of $\alpha(F_i)$ can be calculated by taking the number of all chains that are passing through $F_i$, subtracting those that are passing through both $F_i$ and $H_i$, and finally subtracting those that are passing through both $F_i$ and $G_i$, but not $H_i$. It gives us
$(m-1)!(m+1)!-2((m-1)!)^2-(m-1)!(m-1)(m-1)!=(m!)^2-((m-1)!)^2$.
In this case, there is one red set corresponding to $F_i$, so at least $((m-1)!)^2$ chains are added to $\alpha(F_i)$, making the total number at least $(m!)^2$.

It completes the proof, since now $|\alpha(F)|\ge(m!)^2$ holds for all $F\in\mcf$. Since there are a total of $n!$ chains, it implies $|\mcf|\le \frac{n!}{(m!)^2}=\binom{n}{m}$.
\end{proof}

Note that there is another theorem strongly related to Theorem \ref{tarjanthm}.

\begin{theorem}({\bf Kleitman} \cite{kleitman})
Let $\mcf$ be a family of subsets of $[n]$, where $n\ge 20$. Assume that there are no 3 different subsets in $\mcf$ such that $A=B\cap C$ or $A=B\cup C$.
Then $|\mcf|\le \binom{n}{\frac{n}{2}}$ if $n$ is even, and $|\mcf|\le 2\binom{n-1}{\frac{n-1}{2}}+2$ if $n$ is odd.
\end{theorem}

Now we move over to fork posets.

\begin{theorem}({\bf De Bonis-Katona} \cite{rfork})
Let $V_s$ denote the fork poset, that consists of $s$ unrelated elements and a $s+1$-th one that is smaller than all of the others. Then
$$La(n,V_s)\le\left(1+\frac{2(s-1)}{n}+O\left(\frac{1}{n^2}\right)\right)\binom{n}{\lfloor\frac{n}{2}\rfloor}.$$
\end{theorem}

Now we prove that the same bound (with a weaker error term) stays valid when the forbidden configuration includes that the $s$ unrelated elements must have the same size.

\begin{theorem}\label{sizefork}
Let $\mcf$ be a family of subsets of $[n]$ that contains no $s+1$ different sets such that $B\subset C_1, C_2, \dots C_s$ and $|C_1|=|C_2|=\dots =|C_s|$.
Then
$$|\mcf|\le \left(1+\frac{2(s-1)}{n}+O\left(\frac{\sqrt{\log n}}{n^{3/2}}\right)\right)\binom{n}{\lfloor\frac{n}{2}\rfloor}.$$
\end{theorem}

We need the following two lemmas to prove the theorem.

\begin{lemma} \label{logn} \cite {GLu}
Let $k=2\sqrt{n\log n}$. Then
$$2 \sum_{i=0}^{\lfloor\frac{n}{2}-k\rfloor} {n \choose i}\leq {n\choose \lfloor n/2 \rfloor}\cdot O\left(\frac{1}{n^{3/2}}\right).$$
\end{lemma}

\begin{lemma} \label{sizeforklemma}
Let $\mcf$ be a family satisfying the conditions of Theorem \ref{sizefork}. Let $k=2\sqrt{n\log n}$ and
$$\mcf_k=\left\{F\in\mcf~:~\frac{n}{2}-k\le |F|\le \frac{n}{2}+k\right\}.$$
Then
$$\lambda(\mcf_k)\le 1+\frac{2(s-1)}{n}+O(kn^{-2}).$$
\end{lemma}

\begin{proof}
For all $A\in\mcf$, let $\mcc_A$ denote the set of chains whose smallest element from $\mcf_k$ is $A$. Let $\mcc_0$ denote the set of chains that contain no element of $\mcf_k$. These sets $\mcc_A$ and $\mcc_0$ form a partition of $\mcc$. Obviously
$$\underset{c\in \mcc_{0}}{ave} (|c\cap\mcf_k|)=0.$$

Let $A\in\mcf_k$ be an arbitrary set. Let $X_1, X_2,\dots X_t$ denote the sets from $\mcf_k$ that contain $A$. A random chain between $A$ and $[n]$ meets $X_i$ with a probability of ${n-|A|\choose |X_i|-|A|}^{-1}$. There are no $s$ sets of the same size in $\{X_1, X_2,\dots X_t\}$, so

$$\underset{c\in \mcc_{A}}{ave} (|c\cap\mcf_k|) \le 1+(s-1)\sum_{i=1}^{\frac{n}{2}+k-|A|} {n-|A|\choose i}^{-1}\le 1+\frac{s-1}{\frac{n}{2}-k}+O(n^{-2})\le 1+\frac{2(s-1)}{n}+O(kn^{-2}).$$

We proved the bound for all $\mcc_A$ and also for $\mcc_0$, so it holds for $\mcc$ too.
$$\lambda(\mcf_k)=\underset{c\in \mcc}{ave} (|c\cap\mcf_k|)\le
\max\left( \max_{A\in\mcf_k} \underset{c\in \mcc_A}{ave} (|c\cap\mcf_k|),~\underset{c\in \mcc_0}{ave} (|c\cap\mcf_k|) \right)
\le 1+\frac{2(s-1)}{n}+O(kn^{-2}).$$
\end{proof}

\begin{remark}
In the above proof, we divided the set of all chains into many parts and investigated them separately. This is technique is called the partition method, developed by Griggs, Lu and Li. \cite{diamond} (See also \cite{partition}.) The proofs of Theorem \ref{sizebthm}, Theorem \ref{jthm} and Theorem \ref{sizediathm} will also use a partition method, though the the partitions are defined differently is each case.
\end{remark}

\begin{proof} (of Theorem \ref{sizefork})

By Lemmas \ref{lub2} and \ref{sizeforklemma} we get that
$$|\mcf_k|\le \lambda(\mcf_k){n\choose \lfloor n/2 \rfloor}\le \left(1+\frac{2(s-1)}{n}+O\left(\frac{\sqrt{\log n}}{n^{3/2}}\right)\right){n\choose \lfloor n/2 \rfloor}.$$
Lemma \ref{logn} implies that the number of the remaining sets is negligible compared to that.
$$|\mcf\backslash\mcf_k|\le 2 \sum_{i=0}^{\lfloor\frac{n}{2}-k\rfloor} {n \choose i}\leq {n\choose \lfloor n/2 \rfloor}\cdot O\left(\frac{1}{n^{3/2}}\right).$$
Therefore
$$|\mcf|\le \left(1+\frac{2(s-1)}{n}+O\left(\frac{\sqrt{\log n}}{n^{3/2}}\right)\right){n\choose \lfloor n/2 \rfloor}.$$
\end{proof}

\begin{remark}\label{vremark}
One can get an alternative bound in Theorem \ref{sizefork} that is weaker for large $n$ but contains no unspecified error term. Follow the proof, but define $\mcf_k$ as the family of all sets from $\mcf$ that are different from $[n]$.
$$\underset{c\in \mcc_{A}}{ave} (|c\cap\mcf_k|) \le 1+(s-1)\sum_{i=1}^{n-1-|A|} {n-|A|\choose i}^{-1}=1+(s-1)q(n-|A|).$$
From Lemma \ref{qlemma}, we get that $q(n-|A|)\le\frac{2}{3}$, so
$$\lambda(\mcf_k)\le 1+\frac{2}{3}(s-1).$$
Using Lemma \ref{lub} (with $x=1$, $y=\frac{2}{3}(s-1)$) and $|\mcf\backslash\mcf_k|\le 1$ it follows that
\begin{equation}
|\mcf|\le \binom{n}{\lfloor\frac{n}{2}\rfloor}+\frac{2}{3}(s-1)\binom{n}{\lfloor\frac{n}{2}\rfloor+1}+1.
\end{equation}
\end{remark}

Using our results about fork posets, we can prove upper bounds for batons too.

\begin{notation}
The baton poset $\mcp_h(s,t))$ consists of $h+s+t-2$ elements $A_1,\dots A_s, B_1,\dots B_{h-2}, C_1,\dots C_t$. The relations are $A_1, A_2, \dots , A_s < B_1< B_2 < \dots < B_{h-2} < C_1, C_2, \dots , C_t$.
\end{notation}

\begin{theorem}({\bf Griggs-Lu} \cite{GLu})
$$La(n,\mcp_h(s,t))\le\Sigma(n,h-1)+{n\choose \left\lfloor \frac{n+h}{2}\right\rfloor}\left(\frac{2h(s+t-2)}{n}+O\left(\frac{\sqrt{\log(n)}}{n^{3/2}}\right)\right).$$
\end{theorem}

We strengthen this theorem in two ways. We add size restrictions to the forbidden poset and even under this weaker condition, we prove a stronger bound. (An $h$ can be omitted due to more careful analysis.)

\begin{theorem}\label{restrbaton}
Let $\mcf$ be a family of subsets of $[n]$ that contains no $h+s+t-2$ sets $A_1,\dots A_s, B_1,\dots B_{h-2},$ $C_1,\dots C_t$ such that $A_1, A_2, \dots , A_s \subset B_1\subset B_2\subset \dots \subset B_{h-2}\subset C_1, C_2, \dots , C_t$ and $|A_1|=|A_2|=\dots =|A_s|, ~|C_1|=|C_2|=\dots =|C_t|$. ($s,t\geq 1$, $h\ge 3$.) Then
$$|\mcf|\leq\Sigma(n,h-1)+{n\choose \left\lfloor \frac{n+h}{2}\right\rfloor}\left(\frac{2(s+t-2)}{n}+O\left(\frac{\sqrt{\log(n)}}{n^{3/2}}\right)\right).$$
\end{theorem}

\begin{proof}

Define $k$ and $\mcf_k$ as in Lemma \ref{sizeforklemma}.

Let $\mcf^1$ be the family of those members of $\mcf_k$ that do not contain $s$ other members of the same size from $\mcf_k$. Let $\mcf^2$ be the family of those members of $\mcf_k$, that contain $s$ other members of the same size from $\mcf_k$ and are also contained in $t$ other members of the same size from $\mcf_k$. Let $\mcf^3$ denote the family of the remaining sets of $\mcf_k$. (They contain $s$ other members of the same size from $\mcf_k$, but are not contained in $t$ other members of the same size from $\mcf_k$.) We will give upper bounds on the Lubell functions of these families separately.

There is no chain of $h-2$ sets in $\mcf^2$, otherwise a forbidden configuration would appear. It means that every chain contains at most $h-3$ members of $\mcf^2$, so
$$\lambda(\mcf^2)=\underset{c\in \mcc}{ave}(|c\cap\mcf^2|)\le h-3.$$

There is no member of $\mcf^3$ that contains $t$ other sets of the same size from $\mcf^3$. Lemma \ref{sizeforklemma} implies
$$\lambda(\mcf^3)\le 1+\frac{2(t-1)}{n}+O(kn^{-2}).$$

There is no member of $\mcf^1$ that contains $s$ other sets of the same size from $\mcf^1$. By considering the complements of the sets in $\mcf^1$, an upper bound can be given by Lemma \ref{sizeforklemma}:
$$\lambda(\mcf^1)\le 1+\frac{2(s-1)}{n}+O(kn^{-2}).$$

After adding these bounds, we get
$$\lambda(\mcf_k)=\lambda(\mcf^1)+\lambda(\mcf^2)+\lambda(\mcf^3)\le h-1+\frac{2(s+t-2)}{n}+O(kn^{-2}).$$

Lemma \ref{lub} implies
$$|\mcf_k|\le \Sigma(n,h-1)+{n\choose \left\lfloor \frac{n+h}{2}\right\rfloor}\left(\frac{2(s+t-2)}{n}+O\left(\frac{\sqrt{\log(n)}}{n^{3/2}}\right)\right).$$

Since Lemma \ref{logn} says that
$$|\mcf\backslash\mcf_k|\le {n\choose \lfloor n/2 \rfloor}\cdot O\left(\frac{1}{n^{3/2}}\right),$$
which is negligible compared to the above, the statement of the theorem follows.
\end{proof}

Next, we generalize the following theorem about the butterfly poset $\mathcal{B}$.

\begin{theorem}({\bf De Bonis-Katona-Swanepoel}, \cite{DKS})
Let $\mathcal{B}$ denote the poset that has 4 elements and the relations $A,B<C,D$. Then La$(n,\mathcal{B})=\Sigma(n,2)$.
\end{theorem}

Our theorem gives the same bound when $n$ is large enough even with size restrictions. It is believed that the theorem holds for smaller values, but proving it would require more complicated calculations or case-by-case analysis.

\begin{theorem}\label{sizebthm}
Let $\mcf$ be a family of subsets of $[n]$, where $n\ge 13$. Assume that there are no 4 different subsets $A,B,C,D$ in $\mcf$ such that $A$ and $B$ are both subsets of both $C$ and $D$ and either $|A|=|B|$ or $|C|=|D|$ holds. Then $|\mcf|\le \Sigma(n,2)$.
\end{theorem}

\begin{proof}
Let $\mcf$ be such a family. Assume that $\emptyset\in\mcf$. Then there are no three subsets in $\mcf\backslash\{\emptyset\}$ satisfying $B\subset C$, $B\subset D$ and $|C|=|D|$. Using Remark \ref{vremark} and that $n$ is large enough, we get
$$|\mcf|\le\binom{n}{\lfloor\frac{n}{2}\rfloor}+\frac{2}{3}\binom{n}{\lfloor\frac{n}{2}\rfloor+1}+1+1\le\Sigma(n,2).$$
If $[n]\in\mcf$, then consider the family of the complements of the sets in $\mcf$. It also satisfies the conditions of the theorem, and contains $\emptyset$, so $|\mcf|\le\Sigma(n,2)$ holds in this case too. From now on, we will assume that $\emptyset, [n]\not\in\mcf$.

We will prove that $\lambda(\mcf)\le 2$, then Lemma \ref{lub} (with $x=2$, $y=0$) will imply $|\mcf|\le \Sigma(n,2)$.

Let $\mcg$ denote the set of all members of $\mcf$ that contain an other member of $\mcf$ and are also contained in an other member of $\mcf$. Let $\mcc$ denote the set of all chains. Let $\mcc_0$ denote the set of all chains not containing any member of $\mcg$. For any set $F\in\mcg$, let $\mcc_F$ denote the set of all chains that are passing through $F$, and $F$ is the smallest member of $\mcg$ in them. In this way, the collections of chains $\mcc_0$ and $\mcc_F$ ($F\in\mcg$) form a partition of $\mcc$. It means that
\begin{equation}\label{parteq}|\mcc_0|+\sum_{F\in\mcg}|\mcc_F|=|\mcc|=n!.\end{equation}

Obviously, the chains in $\mcc_0$ contain at most two elements of $\mcf$, so
\begin{equation}\label{c0eq}\sum_{c\in\mcc_0} |c\cap \mcf|\le 2|\mcc_0|.\end{equation}

Let $F\in\mcg$. Let $\mcs_F$ denote the set of all chains passing through $F$ (so $|\mcs_F|=|F|!(n-|F|)!$ and $\mcc_F\subset \mcs_F$). Note that $F$ can not be contained in two sets from $\mcf$ of the same size, since they would form a forbidden configuration together with $F$ and one of its subsets from $\mcf$ ($F$ has a subset like that, since $F\in\mcg$). If $F\subset G$, then $G$ appears in $|F|!(|G|-|F|)!(n-|G|)!=\binom{n-|F|}{|G|-|F|}^{-1}|\mcs_F|$ chains of $\mcs_F$.

Similarly, $F$ can not contain two sets from $\mcf$ of the same size, since they would form a forbidden configuration together with $F$ and one of sets from $\mcf$ that contain $F$. If $G\subset F$, then $G$ appears in $|G|!(|F|-|G|)!(n-|F|)!=\binom{|F|}{|G|}^{-1}|\mcs_F|$ chains from $\mcs_F$. It follows from Lemma \ref{qlemma} iii) that
\begin{equation}\label{mcsineq}\sum_{c\in\mcs_F} |c\cap \mcf|\le (1+q(n-|F|)+q(|F|))|\mcs_F|\le 2|\mcs_F|.\end{equation}

If $C\in\mcs_F\backslash\mcc_F$ then $C$ contains at least two members of $\mcg$. Therefore
\begin{equation}\label{mcscineq}\sum_{c\in\mcs_F\backslash\mcc_F} |c\cap \mcf|\ge 2|\mcs_F\backslash\mcc_F|.\end{equation}

From the inequalities (\ref{mcsineq}) and (\ref{mcscineq}) it follows that
\begin{equation}\label{cfeq}\sum_{c\in\mcc_F}|c\cap \mcf|=\sum_{c\in\mcs_F} |c\cap \mcf|-\sum_{c\in\mcs_F\backslash\mcc_F} |c\cap \mcf|\le 2|\mcs_F|-2|\mcs_F\backslash\mcc_F|=2|\mcc_F|.\end{equation}

Using (\ref{c0eq}), (\ref{cfeq}) and (\ref{parteq}) we get that
$$\sum_{c\in\mcc} |c\cap\mcf|=\sum_{c\in\mcc_0} |c\cap \mcf|+\sum_{F\in\mcg}\left(\sum_{c\in\mcc_F}|c\cap \mcf|\right)\le 2|\mcc_0|+\sum_{F\in\mcg}2|\mcc_F|=2|\mcc|.$$

Using (\ref{lubeq}) we get
$$\lambda(\mcf)=\underset{c\in \mcc}{ave}(|c\cap\mcf|)\le 2,$$
which completes the proof.
\end{proof}

\begin{theorem}({\bf Li} \cite{lithesis})
Let $\mathcal{J}$ denote the poset that consists of 4 elements $A,B,C$ and $D$ such that $A\subset B\subset D$ and $A\subset C$. Then
$La(n,\mathcal{J})\le\Sigma(n,2).$
\end{theorem}

(See \cite{Lbounded} for a general theorem about fan posets, containing the above theorem as a special case.)

We prove that the same bound holds even if the forbidden configuration contains an additional requirement of two sets having the same size.

\begin{theorem}\label{jthm}
Let $\mcf$ be a family of subsets of $[n]$. Assume that there are no 4 different subsets $A,B,C$ and $D$ in $\mcf$ such that $A\subset B\subset D$, $A\subset C$ and $|B|=|C|$.
Then $|\mcf|\le \Sigma(n,2)$.
\end{theorem}

\begin{proof}
It is enough to prove the theorem for even values of $n$, as it follows from $n=2m$ to $n=2m+1$. To see this, assume that we already proved it for even values, and consider an odd $n$. Let $\mcf$ be a sets of subsets of $[n]$, satisfying the conditions of the theorem. Then let
$$\mcf^- = \{F ~|~ F\in\mcf,~1\not\in F\},$$
$$\mcf^+ = \{F\backslash\{1\} ~|~F\in\mcf,~ 1\in F\}.$$

Then $\mcf^-$ and $\mcf^+$ are both families of subsets of $[n-1]$, satisfying the conditions of the theorem. Since $n-1$ is even, their sizes are at most $\Sigma(n-1,2)$. Therefore
$$|\mcf|=|\mcf^-|+|\mcf^+|\le 2\Sigma(n-1,2)=2\left(\binom{n-1}{\frac{n-1}{2}}+\binom{n-1}{\frac{n-1}{2}-1}\right)=2\binom{n}{\frac{n-1}{2}}=
\Sigma(n,2).$$

From now on, we will assume that $n$ is even, and use the notation $m=\frac{n}{2}$. We may also assume that $n\ge 4$, since the statement is trivial for $n=2$.

Assume that $[n]\in\mcf$. Then $\mcf\backslash {[n]}$ contains no three sets such that $A\subset B$, $A\subset C$ and $|B|=|C|$. Remark \ref{vremark} (with $s=2$) and $n\ge 4$ implies that
$$|\mcf|\le\binom{n}{m}+\left\lfloor\frac{2}{3}\binom{n}{m+1}\right\rfloor+1+1\le\Sigma(n,2).$$
From now on, we will assume that $[n]\not\in\mcf$.

In the case of this theorem, $\lambda(\mcf)\le 2$ is not always true, so it is not possible to prove the required bound using the Lubell function. We need a more precise approach. For a set $F\in \mcf$ let $w(F)=\binom{n}{|F|}$ denote the weight of $F$. If $\mcc$ denotes the set of all chains of $[n]$, then
\begin{equation}\label{weightingeq}
\underset{c\in \mcc}{ave}\left(\sum_{F\in c\cap\mcf} w(F)\right)=\frac{1}{n!}\sum_{F\in\mcf} w(F)|F|!(n-|F|)!=\frac{1}{n!}\sum_{F\in\mcf} n!=|\mcf|.\end{equation}

It means that we can give an upper bound to $\mcf$ by analysing the quantity $\underset{c\in \mcc}{ave}\left(\displaystyle\sum_{F\in c\cap\mcf} w(F)\right)$. The following lemma is the key to the proof.

\begin{lemma}\label{jlemma}
Assume that $n=2m$. Let $F$ be a proper subset of $[n]$, and let $\mcd_F$ denote the set of chains between $F$ and $[n]$. Assume that $\mcr$ is a family of subsets of $[n]$ such that there are no 3 different sets $B,C$ and $D$ in $\mcr$ such that $B\subset D$ and $|B|=|C|$. Let us use the notation
$$S(F)=\underset{c\in \mcd_F}{ave}\left(\sum_{\substack{X\in c\cap\mcr \\ F\subsetneq X\subsetneq [n] }} w(X)\right)=\frac{1}{(n-|F|)!}\sum_{c\in \mcd_F}\left(\sum_{\substack{X\in c\cap\mcr \\ F\subsetneq X\subsetneq [n] }} w(X)\right).$$
Then the following inequalities hold:
\begin{enumerate}[i)]
\item If $|F|\ge m-1$, then
$$S(F)\le \binom{n}{|F|+1}.$$
\item If $|F|\le m-1$, then
$$S(F)\le\binom{n}{m}+\sum_{i=|F|+1}^{m-1} \frac{1}{n-i+1}\binom{n}{i}.$$
\end{enumerate}
\end{lemma}

\begin{proof}
The statement is trivially true for $|F|=n-1$ and $|F|=n-2$. For $m-1\le|F|\le n-3$, we will prove the statement by induction, decreasing $|F|$ by 1 at every step.

Let $m-1\le|F|\le n-3$, and assume that we already proved the lemma for the greater values of $|F|$. Let $A_1, A_2, \dots A_{n-|F|}$ be the sets of size $|F|+1$ that contain $F$. We will investigate the families of chains that pass through the sets $A_i$ individually. Let us use the notation $|\{A_1, A_2,\dots, A_{n-|F|}\}\cap \mcr|=N$. Then

\begin{equation}\label{sfformula}S(F)=\frac{N}{n-|F|}\binom{n}{|F|+1}+\frac{1}{n-|F|}\sum_{i=1}^{n-|F|} S(A_i).\end{equation}

By induction, $S(A_i)\le\binom{n}{|F|+2}$ for all $i$. If $A_i, A_j\in\mcr$, then there can't be any set of $\mcr$ that contains $A_i$, so $S(A_i)=S(A_j)=0$. Using these two observations and (\ref{sfformula}), we give an upper bound to $S(F)$.

If $N=0$, then
$$S(F)\le \frac{1}{n-|F|}(n-|F|)\binom{n}{|F|+2}=\binom{n}{|F|+2} \le \binom{n}{|F|+1}.$$

If $N=1$, then
$$S(F)\le\frac{1}{n-|F|}\binom{n}{|F|+1}+\binom{n}{|F|+2}\le \binom{n}{|F|+1}.$$
(The second inequality follows easily from $m-1\le|F|$.)

If $N\ge 2$, then
$$S(F)\le \frac{N}{n-|F|}\binom{n}{|F|+1}+\frac{1}{n-|F|}(n-|F|-N)\binom{n}{|F|+2}\le \binom{n}{|F|+1}.$$

This proves  part i). Now we move on the the proof of part ii). We will use induction again, decreasing $|F|$ by one at every step. The case $|F|=m-1$ was proved already in part i). (The summation is empty in this case.) Assume that $0\le |F|\le m-2$. We can use the same observations as before. The value of $S(A_i)$ can be estimated by induction for all $i$. Additionally, $S(A_i)=S(A_j)=0$, if $i\not=j$ and $A_i,A_j\in\mcr$.

If $N\le 1$, then
$$S(F)\le\frac{1}{n-|F|}\binom{n}{|F|+1}+\binom{n}{m}+\sum_{i=|F|+2}^{m-1} \frac{1}{n-i+1}\binom{n}{i}=
\binom{n}{m}+\sum_{i=|F|+1}^{m-1} \frac{1}{n-i+1}\binom{n}{i}.$$

Now assume that $N\ge 2$. Then $N$ of the values $S(A_i)$ are 0, and the others are at most $$\binom{n}{m}+\displaystyle\sum_{i=|F|+2}^{m-1} \frac{1}{n-i+1}\binom{n}{i}$$ by induction. So
$$S(F)\le \frac{N}{n-|F|}\binom{n}{|F|+1}+\frac{n-|F|-N}{n-|F|} \left(\binom{n}{m}+\sum_{i=|F|+2}^{m-1} \frac{1}{n-i+1}\binom{n}{i}\right).$$

Using the obvious inequality
$$\max\left(\binom{n}{|F|+1},~\binom{n}{m}+\sum_{i=|F|+2}^{m-1} \frac{1}{n-i+1}\binom{n}{i}\right)\le \binom{n}{m}+\sum_{i=|F|+1}^{m-1} \frac{1}{n-i+1}\binom{n}{i},$$

we get that
$$S(F)\le \binom{n}{m}+\sum_{i=|F|+1}^{m-1} \frac{1}{n-i+1}\binom{n}{i}.$$

This completes the proof of part ii).
\end{proof}

Now we continue the proof of Theorem \ref{jthm}. We want to show that
$$\underset{c\in \mcc}{ave}\left(\sum_{F\in c\cap\mcf} w(F)\right)\le \Sigma(n,2).$$
Then (\ref{weightingeq}) will imply $|\mcf|\le\Sigma(n,2)$.

We define a partition of $\mcc$ (the set of all chains). For all sets $A\in\mcf$, let $\mcc_A$ denote the family of chains that pass through $A$, and $A$ is the smallest element of $\mcf$ in them. Additionally, let $\mcc_0$ denote the family of chains that avoid $\mcf$.

Obviously
$$\underset{c\in \mcc_0}{ave}\left(\sum_{F\in c\cap\mcf} w(F)\right)=0.$$

We want to show that
$$\underset{c\in \mcc_A}{ave}\left(\sum_{F\in c\cap\mcf} w(F)\right)\le\Sigma(n,2)$$
for all groups $\mcc_A$.

Let $\mcr=\{G\in\mcf ~|~A\subsetneq G \}$. Then there are no 3 different sets $B,C$ and $D$ in $\mcr$ such that $B\subset D$ and $|B|=|C|$. (Otherwise they would form a forbidden configuration with $A$.) By the definition of $\mcc_A$, its chains do not not contain any sets smaller than $A$. The number of chains in $\mcc_A$ passing through a set $G\in\mcr$ is proportional to the number of chains between $F$ and $[n]$, passing through $G$. Therefore, using the notation from Lemma \ref{jlemma}, we have
$$\underset{c\in \mcc_A}{ave}\left(\sum_{F\in c\cap\mcr} w(F)\right)=S(A).$$

Since all chains of $\mcc_A$ contain $A$,
$$\underset{c\in \mcc_A}{ave}\left(\sum_{F\in c\cap\mcf} w(F)\right)=w(A)+\underset{c\in \mcc_A}{ave}\left(\sum_{F\in c\cap\mcr} w(F)\right)=\binom{n}{|A|}+S(A).$$

If $m-1\le |A|$, then Lemma \ref{jlemma} i) implies
$$\underset{c\in \mcc_A}{ave}\left(\sum_{F\in c\cap\mcf} w(F)\right)\le\binom{n}{|A|}+\binom{n}{|A|+1}\le\Sigma(n,2).$$

Now let $|A|\le m-2$. Lemma \ref{jlemma} ii) implies
$$\underset{c\in \mcc_A}{ave}\left(\sum_{F\in c\cap\mcf} w(F)\right)\le\binom{n}{|A|}+\binom{n}{m}+\sum_{i=|A|+1}^{m-1} \frac{1}{n-i+1}\binom{n}{i}.$$

We have to prove that
$$\binom{n}{|A|}+\binom{n}{m}+\sum_{i=|A|+1}^{m-1} \frac{1}{n-i+1}\binom{n}{i}\le\Sigma(n,2)=\binom{n}{m}+\binom{n}{m-1}.$$

Note that if $i\le m-1$, then
$$\frac{1}{n-i+1}\binom{n}{i}=\frac{n!}{i!(n-i+1)!}\le\frac{n!}{(m-1)!(m+2)!}.$$

So it suffices to prove
$$\binom{n}{|A|}+\binom{n}{m}+\frac{(m-|A|-1)n!}{(m-1)!(m+2)!}\le \binom{n}{m}+\binom{n}{m-1}.$$

After subtracting $\binom{n}{m}$ from both sides and dividing by $n!$, we get
$$\frac{1}{|A|!(n-|A|)!}+\frac{(m-|A|-1)}{(m-1)!(m+2)!}\le \frac{1}{(m-1)!(m+1)!}.$$

After further rearranging, it becomes
$$(m-1)!(m+2)!\le (|A|+3)|A|!(n-|A|)!.$$

Obviously $|A|+1<|A|+3$, so it suffices to prove
$$(m-1)!(m+2)!\le (|A|+1)!(n-|A|)!,$$
or equivalently
$$\binom{n+1}{m-1}^{-1}\le\binom{n+1}{|A|+1}^{-1}.$$

This is true, since $|A|+1\le m-1 <\frac {n+1}{2}$ implies $\binom{n+1}{|A|+1}\le\binom{n+1}{m-1}$.

With this, we proved that $\Sigma(n,2)$ is an upper bound to the average total weight of the intersection of $\mcf$ with a random chain from any $\mcc_A$. Therefore this bound also applies when we consider $\mcc$, since
$$\underset{c\in \mcc}{ave}\left(\sum_{F\in c\cap\mcf} w(F)\right)\le
\max\left(\underset{c\in \mcc_0}{ave}\left(\sum_{F\in c\cap\mcf} w(F)\right),~\max_{A\in\mcf}\underset{c\in \mcc_A}{ave}\left(\sum_{F\in c\cap\mcf} w(F)\right) \right)
\le\Sigma(n,2).$$

Then (\ref{weightingeq}) implies
$$|\mcf|=\underset{c\in \mcc}{ave}\left(\sum_{F\in c\cap\mcf} w(F)\right)\le\Sigma(n,2).$$
\end{proof}

Our last theorem in this section will be about diamond posets.

\begin{notation}
The diamond poset $D_m$ consists of $m+2$ elements such that $A < B_1,~B_2,\dots,B_m < C$.
\end{notation}

The following theorem exactly determines the value of La($n,D_m$) for infinitely many values of $m$. However, for infinitely many values (including $m=2$) it is unknown. (See \cite{groszdia} for the current best bound for $m=2$.)

\begin{theorem}({\bf Griggs-Li-Lu} \cite{diamond})\label{diamondthm}
Let $n,m\geq 2$, and let $t=\lceil\log_2 (m+2)\rceil$.

If $2^{t-1}-1\le m\le 2^t-{t \choose \lfloor t/2 \rfloor} -1$, then
$${\rm La}(n,D_m)=\Sigma(n,t).$$

If $2^t-{t \choose \lfloor t/2 \rfloor}\le m\le 2^t-2$, then
$$\Sigma(n,t)\leq {\rm La}(n,D_m)\leq  \left( t+1-\frac{2^t-m-1}{{t \choose \lfloor t/2 \rfloor}} \right)
 {n \choose \lfloor n/2 \rfloor}.$$
\end{theorem}

Roughly speaking, this theorem tells us that
$$\textrm{La}(n,D_m)=(\log_2 m+O(1)){n \choose \lfloor n/2 \rfloor}.$$

Now we prove that if the forbidden configuration includes that the middle elements must have the same size, the upper bound to the size of the family increases only by a constant factor.

\begin{theorem} \label{sizediathm}
Let $\mcf$ be a family of subsets of $[n]$, and $m\ge 2$. Assume that there are no $m+2$ different subsets $A,~B_1,~B_2,\dots B_m,~C\in\mcf$ such that $A\subset B_i\subset C$ for all $1\le i\le m$ and $|B_1|=|B_2|=\dots=|B_m|$. Then $|\mcf|\le 3(\lceil\log_3 (m-1)\rceil+1)\cdot\binom{n}{\lfloor\frac{n}{2}\rfloor}$.
\end{theorem}

\begin{proof}
Let us use the notation $K=\lceil\log_3 (m-1)\rceil+1$. We will prove that $\lambda(\mcf)\le 3K$, then Lemma \ref{lub2} will imply the statement of the theorem.

Now we partition $\mcc$ (the set of all chains of $[n]$) into some sets. Let $F,G\in\mcf$ be two sets such that $F\subsetneq G$. Let $\mcc_{FG}$ denote the set of chains whose smallest intersection with $\mcf$ is $F$ and the largest one is $G$. Let $\mcc_0$ denote the set of chains that contain at most 1 element of $\mcf$. Then every chain is in exactly one of these sets.

We will prove that in every $\mcc_{FG}$ (and also in $\mcc_0$), the chains contain at most $3K$ elements of $\mcf$ on average. Then all chains contain at most $3K$ elements of $\mcf$ on average, in other words $\lambda(\mcf)\le 3K$. This is obviously true for $\mcc_0$, since its chains contain at most 1 element of $\mcf$.

Now let $F,G\in\mcf$ be two sets such that $F\subset G$. Assume $\mcc_{FG}$ is not empty, and consider the chains in it. These chains pass through $F$ and $G$ and possibly some sets that contain $F$ and are contained in $G$. For any $|F|<k<|G|$ there are $\binom{|G|-|F|}{k-|F|}$ sets satisfying $F\subset X\subset G$, but at most $m-1$ of them can be in $\mcf$, otherwise we would get a forbidden configuration. Therefore the average number of sets from $\mcf$ contained in the chains of $\mcc_{FG}$ is at most
\begin{equation}\label{Dmeq}
\underset{c\in \mcc_{FG}}{ave}(|c\cap\mcf|)\le 2+\sum_{i=1}^{|G|-|F|-1} \min\left(\frac{m-1}{\binom{|G|-|F|}{i}},1\right).
\end{equation}
After introducing the notation $N=|G|-|F|$ and moving the 2 inside the summation it becomes
$$\underset{c\in \mcc_{FG}}{ave}(|c\cap\mcf|)\le \sum_{i=0}^{N} \min\left(\frac{m-1}{\binom{N}{i}},1\right).$$
There are $N+1$ terms and all of them are at most 1, so the sum is at most $N+1$. This fact finishes the proof when $N<3K$. From now on, we will assume that $N\ge 3K$.

The sum of the first $K$ and the last $K$ summands is obviously at most $2K$. We will show that the rest of the terms are sufficiently small. Assume that $K\le i\le N-K$. Then
$$\frac{m-1}{\binom{N}{i}}\le \frac{m-1}{\binom{N}{K}}\le
\frac{m-1}{(\frac{N}{K})^{K}}=\frac{K}{N}\cdot\frac{m-1}{(\frac{N}{K})^{K-1}}\le\frac{K}{N}\cdot\frac{m-1}{3^{\log_3 (m-1)}}=\frac{K}{N}.$$

So the sum of the middle terms is at most $(N+1-2K)\cdot\frac{K}{N}\le K$. Therefore
$$\underset{c\in \mcc_{FG}}{ave}(|c\cap\mcf|)\le 3K.$$

Since this holds for all $\mcf_{FG}$ and also for $\mcc_0$, we get
$$\lambda(\mcf)=\underset{c\in \mcc}{ave}(|c\cap\mcf|)\le 3K.$$
\end{proof}

\begin{theorem}
Let $m=4$ and $n\ge 3$ in the above theorem. Then $|\mcf|\le\Sigma(n,4)$ and this bound is the best possible.
\end{theorem}

\begin{proof}
Consider formula $(\ref{Dmeq})$. We will give an elementary upper bound using Lemma \ref{qlemma} ii).
$$\underset{c\in \mcc_{FG}}{ave}(|c\cap\mcf|)\le 2+\sum_{i=1}^{N-1} \min\left(\frac{m-1}{\binom{N}{i}},1\right)\le 2+(m-1)\cdot q(N)\le 2+\frac{2}{3}(m-1)=4.$$
This leads to
$$\lambda(\mcf)=\underset{c\in \mcc}{ave}(|c\cap\mcf|)\le 4,$$
and Lemma \ref{lub} implies
$$|\mcf|\le\Sigma(n,4).$$
It is easy to see that the forbidden configuration does not appear is the family that consists of all subsets of the 4 middle levels, therefore the bound is the best possible.
\end{proof}

\begin{remark}
So far, the results in the size restricted problems were equal or almost equal to their counterparts without size restrictions. However, this is not true for diamond posets.

The answer found in the above theorem is different from the one for the same problem without size restrictions. Substituting $m=4$ to Theorem \ref{diamondthm}, we get that the best possible bound is
$$|\mcf|\le\Sigma(n,3).$$

For general $m$, Theorem \ref{diamondthm} implies that the answer is
$$\textrm{La}(n,D_m)=\Sigma(n, \log_2 m+O(1))$$
in the simple case. In the size restricted case, Theorem \ref{sizediathm} gives the upper bound
$$|\mcf|\le 3(\lceil\log_3 (m-1)\rceil+1)\cdot\binom{n}{\lfloor\frac{n}{2}\rfloor}.$$
Now we construct a large family $|\mcf|$ that does not contain $D_m$ with size restrictions. Let $r$ be the largest integer such that $\binom{r}{\lfloor\frac{r}{2}\rfloor}<m$, and let $\mcf$ consist of all subsets of $[n]$ in the $r$ middle levels. Using Stirling's formula, it follows that $r=\log_2 m+O(\log_2 \log_2 m)$, therefore
$$|\mcf|=\Sigma(n,r)=\Sigma(n, \log_2 m+O(\log_2 \log_2 m)).$$
\end{remark}

\section{A general bound}

In this section we prove a general theorem about forbidden poset problems with size restrictions. It was motivated by the following result about induced subposets.

\begin{theorem}({\bf Methuku-P\'alv\"{o}lgyi}, \cite{indgeneral})
For every finite poset $\mcp$, there exists a constant $C$ such that
$$La^\star(n,\mcp)\le C\binom{n}{\lfloor\frac{n}{2}\rfloor}.$$
\end{theorem}

Now let us add size restrictions instead of the induced property. We prove that the bound $C\binom{n}{\lfloor\frac{n}{2}\rfloor}$ applies in this case too. (The theorems are independent, neither one implies the other.)

\begin{theorem}\label{genthm}
Let $\mcp$ be a finite poset. The elements of $\mcp$ are colored with the colors $1, 2,\dots, k$. (Each element has exactly one color and all colors are used.) Assume that the coloring is order-preserving. (If $a<_p b$, then $a$'s color is smaller than $b$'s color.) Then there exists a constant $C$ (depending on $\mcp$ and its coloring) such that for any family $\mcf$ of subsets of $[n]$ satisfying $|\mcf|>C\binom{n}{\lfloor\frac{n}{2}\rfloor}$, there is an embedding $f:\mcp\rightarrow\mcf$ that maps all elements of the same color into sets of the same size.
\end{theorem}

The above theorem is an easy consequence of the following lemma.

\begin{lemma}\label{genlemma}
Let $a_1, a_2, \dots a_k$ be given positive integers. Then there exist a constant $C(a_1, a_2,\dots a_k)$ such that the following holds for every $n\in\mathbb{N}$ and family $\mcf$ of subsets of $n$ satisfying $|\mcf|>C(a_1, a_2,\dots a_k)\binom{n}{\lfloor\frac{n}{2}\rfloor}$.

One can find $2+\displaystyle\sum_{i=1}^k a_i$ different sets called $X, Y_1^1, Y_2^1,\dots Y_{a_1}^1, Y_1^2,\dots Y_{a_k}^k$ and $Z$ in $\mcf$ such that
\begin{enumerate}[i)]
\item $|Y_1^i|=|Y_2^i|=\dots=|Y_{a_i}^i|$ for all $1\le i\le k$.
\item If $1\le i_1<i_2\le k$, then $Y_{j_1}^{i_1}\subset Y_{j_2}^{i_2}$ for all $1\le j_1\le a_{i_1}$, $1\le j_2\le a_{i_2}$.
\item $X\subset A_j^1$ for all $1\le j\le a_1$.
\item $A_j^k\subset Z$ for all $1\le j\le a_k$.
\end{enumerate}
\end{lemma}

\begin{proof}
The lemma will be proved by induction on $k$. If $k=1$, then the statement directly follows from Theorem \ref{sizediathm}, since we are looking for $a_k+2$ sets forming a diamond poset with size restrictions. Now assume that $k\ge 2$ and we already proved the lemma for smaller values of $k$.

Let $\mcg$ denote the set of those sets $G\in\mcf$ for which there are $a_k+1$ another members of $\mcf$ (named $Q_1, Q_2\dots Q_{a_k}$ and $T$) such that $|Q_1|=|Q_2|=\dots=|Q_{a_k}|$ and $G\subset Q_i\subset T$ for all $1\le i\le a_k$.

Then the diamond poset $D_{a_k}$ can not be embedded into $\mcf\backslash\mcg$ in a way that the middle elements are mapped into sets of the same size. (Otherwise the set corresponding to the bottom element of the diamond would belong in $\mcg$.) Theorem \ref{sizediathm} implies that $|\mcf\backslash\mcg|\le C'\binom{n}{\lfloor\frac{n}{2}\rfloor}$, where $C'$ depends only on $a_k$.

Let $C(a_1, a_2,\dots a_k)=C'+C(a_1, a_2,\dots a_{k-1})$. If $|\mcf|> C(a_1, a_2,\dots a_k)\binom{n}{\lfloor\frac{n}{2}\rfloor}$, then we have $|\mcg|>C(a_1, a_2,\dots a_{k-1})\binom{n}{\lfloor\frac{n}{2}\rfloor}$. By induction, one can find $2+\displaystyle\sum_{i=1}^{k-1} a_i$ sets in $\mcg$ satisfying the conditions of the lemma. The largest of these sets, $Z$ is also in $\mcg$, so there are some sets $Q_1, Q_2\dots Q_{a_k}, T\in\mcf$ such that $|Q_1|=|Q_2|=\dots=|Q_{a_k}|$ and $Z\subset Q_i\subset T$ for all $1\le i\le a_k$. By renaming $Q_j$ to $Y_j^k$ for all $1\le j\le a_k$, removing $Z$, and picking $T$ as the new $Z$, we found a configuration of $2+\displaystyle\sum_{i=1}^{k} a_i$ sets satisfying the conditions of the lemma.
\end{proof}

\begin{proof} (of Theorem \ref{genthm})
Let $a_i$ denote the number of elements that are colored with $i$. Let $\mcf$ be a family of subsets of $[n]$ such that $|\mcf|>C\binom{n}{\lfloor\frac{n}{2}\rfloor}=C(a_1, a_2,\dots a_k)\binom{n}{\lfloor\frac{n}{2}\rfloor}$. Consider the sets $Y_j^i$ given by Lemma \ref{genlemma}. Let $f:\mcp\rightarrow\mcf$ be a function that maps the elements of color $i$ to the sets $Y_1^i, Y_2^i,\dots, Y_{a_i}^i$ in an arbitrary order.

Then $f$ will obviously satisfy the requirements. Lemma \ref{genlemma} i) means that the elements having the same color are mapped into sets of the same size. If $a,b\in\mcp$ and $a<_p b$, then the color of $a$ must be smaller than the color of $b$. Lemma \ref{genlemma} ii) implies that $f(a)\subset f(b)$, so $f$ is an embedding.

Note that since every poset $\mcp$ has a finite number of possible colorings, we can pick a constant $C$ that depends only on $\mcp$ and not on the coloring.
\end{proof}

{\bf Acknowledgement} I would like to thank Gyula O.H. Katona for his help with the creation of this paper. This research was supported by National Research, Development and Innovation Office - NKFIH, grant number K116769.


\begin{thebibliography}{99}

\bibitem{indtree} E. Boehnlein and T. Jiang, Set Families With a Forbidden Induced Subposet, {\it Combinatorics, Probability and Computing} {\bf 21} (2012) 492-511.

\bibitem{bukh} B. Bukh, Set families with a forbidden subposet, {\it Electronic J. of Combinatorics} {\bf 16} (2009) R142, 11p.

\bibitem{burcsi} P. Burcsi, D. T. Nagy, The method of double chains for largest families with excluded subposets {\it Electronic Journal of Graph Theory and Applications} {\bf 1} (2013) 40-49.

\bibitem{chen} H.-B. Chen and W.-T. Li, A Note on the Largest Size of Families of Sets with a Forbidden Poset, {\it Order} {\bf 31} (2014) 137-142.

\bibitem{rfork} A. De Bonis and G. O. H. Katona, Largest families without an r-fork, {\it Order} {\bf 24} (2007) 181-191.

\bibitem{DKS} A. De Bonis, G. O. H. Katona and K. J. Swanepoel, Largest family without $A\cup B\subseteq C\cap D$, {\it J. Combinatorial Theory (Ser A)} {\bf 111} (2005) 331-336.

\bibitem{diamond} J. R. Griggs, W.-T. Li and L. Lu, Diamond-free families, {\it J. Combinatorial Theory (Ser A)} {\bf 119} (2012) 310-322.

\bibitem{partition} J. R. Griggs and W.-T. Li, The partition method for poset-free families, {\it J. Combinatorial Optimization} {\bf 25} (2013), 587-596.

\bibitem{Lbounded} J. R. Griggs and W.-T. Li, Poset-free families and Lubell-boundedness, {\it J. Combinatorial Theory (Ser A)} {\bf 134} (2015) 166-187.

\bibitem{GLu} J. R. Griggs and L. Lu, On families of subsets with a forbidden subposet, {\it Combinatorics, Probability, and Computing} {\bf 18} (2009) 731-748.

\bibitem{grosz} D. Gr\'osz, A. Methuku and C. Tompkins, An improvement of the general bound on the largest family of subsets avoiding a subposet {\it to appear in Order} (2016)

\bibitem{groszdia} D. Gr\'osz, A. Methuku and C. Tompkins, An upper bound on the size of diamond-free families of sets, {\it arXiv:1601.06332}

\bibitem{tarjan} G. O. H. Katona and T. G. Tarj\'an, Extremal problems with excluded subgraphs in the n-cube, {\it Lecture Notes in Math.} {\bf 1018} (1981) 84-93.

\bibitem{lithesis} W.-T. Li, Extremal problems on families of subsets with forbidden subposets, PhD dissertation, University of South Carolina, (2011)

\bibitem{Lub} D. Lubell, A short proof of Sperner's lemma, {\it J. Combinatorial Theory} {\bf 1}, (1966), 299.

\bibitem{kleitman} D. Kleitman, Collections  of  Sets without Unions and Intersections {\it Proceedings of the third international conference on Combinatorial mathematics} (1989) 272-282.

\bibitem{indgeneral} A. Methuku and D. P\'alv\"{o}lgyi, Forbidden hypermatrices imply general bounds on induced forbidden subposet problems, {\it arXiv:1408.4093}
\end{thebibliography}
\end{document}